\RequirePackage{lineno}
\documentclass[12pt,a4paper,leqno,verbatim]{amsart}

\newcounter{minutes}
\divide\time by 60
\newcounter{hours}
\multiply\time by 60 \addtocounter{minutes}{-\time}

\usepackage{amssymb}
\usepackage{hyperref}
\usepackage{graphicx}
\date{}
\newfont{\cyrilic}{wncyr10 scaled 1000}
\title{Power mean inequality of generalized trigonometric functions}
\author[Barkat Ali Bhayo]{Barkat Ali Bhayo}
\email{barbha@utu.fi}
\author[Matti Vuorinen]{Matti Vuorinen}
\address{Department of Mathematics \& Statistics, University of Turku,
FI-20014 Turku, Finland}
\email{vuorinen@utu.fi}
\newcommand{\comment}[1]{}

\swapnumbers
\theoremstyle{plain}

\newtheorem{theorem}[equation]{Theorem}
\newtheorem{lemma}[equation]{Lemma}

\newtheorem{subsec}[equation]{}

\newtheorem{proposition}[equation]{Proposition}

\numberwithin{equation}{section}

\begin{document}
\font\fFt=eusm10 
\font\fFa=eusm7  
\font\fFp=eusm5  
\def\K{\mathchoice
{\hbox{\,\fFt K}}
{\hbox{\,\fFt K}}
{\hbox{\,\fFa K}}
{\hbox{\,\fFp K}}}

\def\thefootnote{}
\footnotetext{ \texttt{\tiny File:~\jobname .tex,
          printed: \number\year-\number\month-\number\day,
          \thehours.\ifnum\theminutes<10{0}\fi\theminutes}
} \makeatletter\def\thefootnote{\@arabic\c@footnote}\makeatother
\maketitle

\begin{abstract} We study the convexity/concavity properties of the generalized
$p$-trigonometric functions in the sense of P. Lindqvist with respect to the power means.
\end{abstract}

\bigskip

\section{Introduction}

The generalized trigonometric and hyperbolic functions depending on
a parameter $p>1$ were studied by P. Lindqvist in 1995 \cite{lp}.
For the case when $p=2$, these functions coincide with elementary
functions. Later on numerous authors have extended this work in
various directions see \cite{be, bem0, bem, dm, lpe}.

For $t\in\mathbb{R}$ and $x,y>0$, the Power Mean $M_t$ of order $t$ is defined by

$$M_t=\displaystyle\left\{\begin{array}{lll} \displaystyle\left(\frac{x^t+y^t}{2}\right)^{1/t},\;\quad t\neq 0,\\

\sqrt{x\,y}, \quad\qquad\qquad t=0\,.\end{array}\right.$$

\begin{theorem}\label{thm1} For $p>1,\,t\geq 0$ and $r,s\in(0,1)$, we have
\begin{enumerate}
\item ${\rm arcsin}_{p}(M_t(r,s))\leq M_t({\rm arcsin}_{p}(r),{\rm arcsin}_{p}(s))$\,,\\
\item ${\rm artanh}_p(M_t(r,s))\leq M_t({\rm artanh}_p(r),{\rm artanh}_p(s))\,,$\\
\item ${\rm arctan}_p(M_t(r,s))\geq M_t({\rm arctan}_p(r),{\rm arctan}_p(s))\,,$\\
\item ${\rm arsinh}_{p}(M_t(r,s))\geq M_t({\rm arsinh}_{p}(r),{\rm arsinh}_{p}(s))$\,.
\end{enumerate}
\end{theorem}

\begin{theorem}\label{thm2} For $p>1,\,t\geq 1$ and $r,s\in(0,1)$, the following relations hold
\begin{enumerate}
\item ${\rm sin}_{p}(M_t(r,s))\geq M_t({\rm sin}_{p}(r),{\rm sin}_{p}(s))$\,,\\
\item ${\rm cos}_{p}(M_t(r,s))\leq M_t({\rm cos}_{p}(r),{\rm cos}_{p}(s))$\,,\\
\item ${\rm tan}_p(M_t(r,s))\leq M_t({\rm tan}_p(r),{\rm tan}_p(s))\,,$\\
\item ${\rm tanh}_p(M_t(r,s))\geq M_t({\rm tanh}_p(r),{\rm tanh}_p(s))\,,$\\
\item ${\rm sinh}_{p}(M_t(r,s))\leq M_t({\rm sinh}_{p}(r),{\rm arsinh}_{p}(s))$\,.
\end{enumerate}
\end{theorem}

In \cite{bv1}, there are some results which are the special case of the above theorems when $t=0$ and $t=2$.

Generalized convexity/concavity with respect to general mean values has been studied
   recently in \cite{avv1}. 

Let $f:I\to (0,\infty)$ be continuous, where $I$ is a subinterval of $(0,\infty)$. Let $M$
and $N$ be any two mean values. We say that  $f$ is $MN$-convex (concave) if
$$f (M(x, y)) \leq (\geq)
N(f (x), f (y)) \,\, \text{ for \,\, all} \,\, x,y \in I\,.$$

In conclusion, we see that the above results are $(M_t, M_t)$-convexity
or $(M_t, M_t)$-concavity properties
of the functions involved. In view of \cite{avv1}, it is natural to expect
that similar results
might also hold for some other pairs  $(M,N)$ of mean values.
\section{Preliminaries}

We introduce some
notation and terminology for the satatement of the main results.

Given complex numbers $a,b$ and $c$ with $c\neq0,-1,-2,\ldots$,
the \emph{Gaussian hypergeometric function} is the
analytic continuation to the slit place $\mathbb{C}\setminus[1,\infty)$ of the series
$$F(a,b;c;z)={}_2F_1(a,b;c;z)=\sum^\infty_{n=0}\frac{(a,n)(b,n)}
{(c,n)}\frac{z^n}{n!},\qquad |z|<1.$$
Here $(a,0)=1$ for $a\neq 0$, and $(a,n)$ is the \emph{shifted factorial function}
 or the \emph{Appell symbol}
$$(a,n)=a(a+1)(a+2)\cdots(a+n-1)$$
for $n\in \mathbb{Z}_+$, see \cite{AS}.

Lets start the
discussion of eigenfunctions of one-dimensional
 $p$-Laplacian $\Delta_p$ on $(0,1)$,\,
  $p\in(1,\infty).$ The eigenvalue problem \cite{dm}
$$
-\Delta_p u=-\left(|u^{'}|^{p-2}u^{'}\right)^{'}
=\lambda|u|^{p-2}u,\quad u(0)=u(1)=0,
$$
has eigenvalues
$$\lambda_n=(p-1)(n \pi_p)^p,\,$$
and eigenfunctions
$$\sin_p(n \pi_p\, t),\quad n\in\mathbb{N},\,$$
where $\sin_p$ is the inverse function of ${\rm arcsin}_p$\,, which is
defined below, and
$$\pi_p=\frac{2}{p}\int^1_0(1-s)^{-1/p}s^{1/p-1}ds=\frac{2}
{p}\,B\left(1-\frac{1}{p},\frac{1}{p}\right)=\frac{2 \pi}{p\,\sin(\pi/p)}\,,$$
with $\pi_2=\pi$.

Lets consider the following
homeomorphisms
\begin{equation*}
\sin_p:(0,a_p)\to I,\quad\cos_p:(0,a_p)\to I,\quad\tan_p:(0,b_p)\to I,\,\\
\end{equation*}
\begin{equation*}
\sinh_p:(0,c_p)\to I, \quad \tanh_p:(0,\infty)\to I\,,
\end{equation*}
where $I=(0,1)$ and
$$a_p=\frac{\pi_p}{2},
\,b_p=2^{-1/p}
F\left(\frac{1}{p},\frac{1}{p};1+\frac{1}{p};\frac{1}{2}\right),\,
c_p=F\left(1\,,\frac{1}{p}
;1+\frac{1}{p}\,;1\right)\,.$$

For $x\in I$, their inverse functions are defined as
\begin{eqnarray*}
{\rm arcsin}_p\,x&=&\int^x_0(1-t^p)^{-1/p}dt=x\,F\left(\frac{1}{p},
\frac{1}{p};1+\frac{1}{p};x^p\right)\\
                 &=&x(1-x^p)^{(p-1)/p}F\left(1,1;1+\frac{1}{p};x^p\right)\,,\\
{\rm arctan}_p\,x&=&\int^x_0(1+t^p)^{-1}dt=x \,
F\left(1,\frac{1}{p};1+\frac{1}{p};-x^p\right)\,\\
&=& \left(\frac{x^p}{1+x^p}\right)^{1/p}F\left(\frac{1}{p},\frac{1}{p};1+\frac{1}{p};\frac{x^p}{1+x^p}\right)\,,\\
{\rm arsinh}_p\,x&=&\int^x_0(1+t^p)^{-1/p}dt=
x\,F\left(\frac{1}{p}\,,\frac{1}{p};1+\frac{1}{p};-x^p\right)\\
&=& \left(\frac{x^p}{1+x^p}\right)^{1/p}F\left(1,\frac{1}{p};1+\frac{1}{p};\frac{x^p}{1+x^p}\right)\,,\\
{\rm artanh}_p\,x&=&\int^x_0(1-t^p)^{-1}dt=x\,F\left
(1\,,\frac{1}{p};1+\frac{1}{p};x^p\right)\,,
\end{eqnarray*}
and by \cite[Prop 2.2]{be}  ${\rm arccos}_p\,x={\rm arcsin}_p((1-x^p)^{1/p})$.
In particular, they
reduce to the familiar functions for the case $p=2$.

The above functions were generalized in two parameters $(p,q)$ in \cite{bv1,t,egl}.

For $x\in I=[0,1]$
$$
{\rm arcsin}_{p,q}\,x=\int^x_0(1-t^q)^{-1/p}dt=
x\,F\left(\frac{1}{p},\frac{1}{q};1+\frac{1}{p};x^q\right).$$
We also define
${\rm arccos}_{p,q}\,x={\rm arcsin}_{p,q}((1-x^p)^{1/q})$
(see \cite[Prop. 3.1]{egl}),
and
$${\rm arsinh}_{p,q}\,x=\int^x_0(1+t^q)^{-1/p}dt=
x\,F\left(\frac{1}{p},\frac{1}{q};1+\frac{1}{q};-x^q\right).$$
Their inverse functions are
$$\sin_{p,q}:(0,\pi_{p,q}/2)\to (0,1),\quad \cos_{p,q}:(0,\pi_{p,q}/2)\to (0,1), $$
$$\sinh_{p,q}:(0,n_{p,q})\to (0,1),\quad
n_{p,q}=\frac{1}{2^{1/p}}
F\left(1,\frac{1}{p};1+\frac{1}{q};\frac{1}{2}\right).$$

For easy reference we record the following lemma from \cite{avvb},
sometimes which is called the \emph{monotone l'Hospital rule}.
\begin{lemma}\label{125}\cite[Theorem 1.25]{avvb}
For $-\infty<a<b<\infty$,
let $f,g:[a,b]\to \mathbb{R}$
be continuous on $[a,b]$, and be differentiable on
$(a,b)$. Let $g^{'}(x)\neq 0$
on $(a,b)$. If $f^{'}(x)/g^{'}(x)$ is increasing
(decreasing) on $(a,b)$, then so are
$$[f(x)-f(a)]/[g(x)-g(a)]\quad and \quad [f(x)-f(b)]/[g(x)-g(b)].$$
If $f^{'}(x)/g^{'}(x)$ is strictly monotone,
then the monotonicity in the conclusion
is also strict.
\end{lemma}

For the next two lemmas see \cite[Theorems 1.1, 1.2, 2.5 \& Lemma 3.6]{bv}.

\begin{lemma}\label{bvthm1} For $p>1$ and $x\in(0,1)$, we have\\

\begin{enumerate}
\item $\left(1+\displaystyle\frac{x^p}{p(1+p)}\right)x<
{\rm arcsin}_p\,x< \displaystyle\frac{\pi_p}{2}\,x$,\\

\item $\left(1+\displaystyle\frac{1-x^p}{p(1+p)}\right)(1-x^p)^{1/p}
< {\rm arccos}_p\,x< \displaystyle\frac{\pi_p}{2}\,
(1-x^p)^{1/p}$,\\

\item $\displaystyle\frac{(p(1+p)(1+x^p)+x^p)x}{p(1+p)(1+x^p)^{1+1/p}}<
{\rm arctan}_p\,x
< 2^{1/p}\,\,b_p\,
\left(\displaystyle\frac{x^p}{1+x^p}\right)^{1/p}$.\\

\item $z\left (1+\displaystyle\frac{\log(1+x^p)}{1+p}\right)
< {\rm arsinh}_p\,x
< z\left(1+\frac{1}{p}\log(1+x^p)\right),
\,\,z=\displaystyle\left(\frac{x^p}{1+x^p}\right)^{1/p},$\\

\item
$x\left(1-\displaystyle\frac{1}{1+p}\log(1-x^p)\right)< {\rm artanh}_p\,x
< x\left(1-\frac{1}{p}\log(1-x^p)\right)\,.$
\end{enumerate}
\end{lemma}

\begin{lemma}\label{bvthm2} For $p,q>1$ and $r,s\in(0,1)$, the following inequalities hold:
\\
\begin{enumerate}
\item ${\rm arcsin}_p(\sqrt{r\,s})\leq \sqrt{{\rm arcsin}_p(r)\,{\rm arcsin}_p(s)}\,,$\\

\item ${\rm artanh}_p(\sqrt{r\,s})\leq \sqrt{{\rm artanh}_p(r)\,{\rm artanh}_p(s)}\,,$\\

\item $\sqrt{{\rm arsinh}_p(r)\,{\rm arsinh}_p(s)}\leq
{\rm arsinh}_p(\sqrt{r\,s})\,,$\\

\item $\sqrt{{\rm arctan}_p(r)\,{\rm arctan}_p(s)}\leq{\rm arctan}_p(\sqrt{r\,s})\,,$\\

\item $\pi_{\sqrt{p\,q}}\leq \sqrt{\pi_p\,\pi_q}$.

\end{enumerate}
\end{lemma}

\begin{lemma}\label{lem1} For $m\geq -1,\,p>1$, the following functions
\begin{enumerate}
\item $f_1(x)=\displaystyle\left(\frac{{\rm arcsin}_p\,x}{x}\right)^{m}\frac{d}{dx}({\rm arcsin}_p\,x)$,\\
\item $f_2(x)=\displaystyle\left(\frac{{\rm artanh}_p\,x}{x}\right)^{m}\frac{d}{dx}({\rm artanh}_p\,x)$,\\
are increasing in $x\in(0,1)$, and
\item $f_3(x)=\displaystyle\left(\frac{{\rm arctan}_p\,x}{x}\right)^{m}\frac{d}{dx}({\rm arctan}_p\,x)$,\\
\item $f_4(x)=\displaystyle\left(\frac{{\rm arsinh}_p\,x}{x}\right)^{m}\frac{d}{dx}({\rm arsinh}_p\,x)$,\\
\end{enumerate}
are decreasing in $x\in(0,1)$.
\end{lemma}

\begin{proof} By definition
$$f_1(x)=\displaystyle\left(\frac{{\rm arcsin}_p\,x}{x}\right)^{m}\frac{1}{(1-x^p)^{1/p}}.$$
For $m\geq 0$, $\left(\frac{{\rm arcsin}_p\,x}{x}\right)^{m}$ is increasing by Lemma \ref{125},
and clearly $(1-x^p)^{1/p}$ is increasing. For the case $m\in[-1,0)$, we define
$$h_1(x)=\left(\frac{x}{{\rm arcsin}_p\,x}\right)^{s}\frac{1}{(1-x^p)^{1/p}},\quad\,s\in(0,1].$$
We get
\begin{eqnarray*}
h'_1(x)&=& \xi((1-x^p)^{1/p}(x^p+s(1-x^p))F_1(x)-s(1-x^p))\\
&>&\xi\left((1-x^p)^{1/p}(x^p+s(1-x^p))
(1+\frac{x^p}{p(1+p)})-s(1-x^p)\right)>0,
\end{eqnarray*}
by Lemma \ref{bvthm1}(1), where
$$\xi=\frac{(1-x^p)^{-(1+2/p)}}{x}\left(\frac{1}{F_1(x)}\right)^{1+s} \quad {\rm and}\quad F_1(x)=F\left(\frac{1}{p},
\frac{1}{p};1+\frac{1}{p};x^p\right).$$

For (2), clearly $f_2$ is increasing for $m\geq 0$. For the case when $ m\in [-1,0)$, we define
$$h_2(x)=\left(\frac{x}{{\rm artanh}_p\,x)}\right)^s\frac{1}{1-x^p},\quad s\in(0,1].$$
Differentiating with respect to $x$, we get
$$h'_2(x)=\frac{\left(F_2(x)\right)^{-(1+s)} \left(\left(p x^p-s x^p+s\right) \,
   F_2(x)-s\right)}{x \left(x^p-1\right)^2}>0,$$
where
$F_2(x)=F\left(1,\frac{1}{p};1+\frac{1}{p};x^p\right)$.

 For (3), the proof for the
case when $m\geq 0$ follows similarly from Lemma \ref{125}. For the
case $m\in[-1,0)$, let
$$h_3(x)=\displaystyle\left(\frac{{\rm arctan}_p\,x}{x}\right)^{-s}\frac{d}{dx}({\rm arctan}_p\,x),\quad\,s\in(0,1].$$
We have
\begin{eqnarray*}
h'_3(x)&=&\frac{F_3(x)^{-(1+s)}}{r(1+r^p)^2}((s+s\,r^p-p\,r^p)F_3(x)-s)\\
&<& \frac{F_3(x)^{-(1+s)}}{r(1+r^p)^2}((s+s\,r^p-s\,r^p)F_3(x)-s)\\
&=& -\frac{s\,F_3(x)^{-(1+s)}}{r(1+r^p)^2}(1-F_3(x))<0,
\end{eqnarray*}
where $F_3(x)=F\left(1,\frac{1}{p};1+\frac{1}{p};-x^p\right)$

For (4), when $m\geq 0$, the proof follows from Lemma \ref{125}. For $m\in[-1,0)$, let
$$h_4(x)=\left(\frac{x}{{\rm arsinh}_p\,x}\right)^{s}\frac{1}{(1+x^p)^{1/p}},\quad s\in(0,1]\,.$$
We have
\begin{eqnarray*}
h_4'(x)&=& \gamma((1-x^p)^{1/p}(s(1-x^p)-x^p)F_4(x)-s(1-x^p))\\
&<&\gamma\left(s(1+x^p)\left(1+\frac{1}{p}\log(1+x^p)\right)-s(1+x^p)-
x^p\left(1+\frac{1}{1+p}\log(1+x^p)\right)\right)\\
&=& \frac{\gamma}{p(1+p)}(s(1+x^p)(1+p)\log(1+x^p)-p(1+p)x^p-p\,x^p\log(1+x^p))\\
& <&0,
\end{eqnarray*}
by Lemma \ref{bvthm1}(4), where
$$\gamma=\frac{(1-x^p)^{-(1+2/p)}}{x}\left(\frac{1}{F_4(x)}\right)^{1+s}\quad {\rm and}\quad  F_4(x)=F\left(\frac{1}{p},
\frac{1}{p};1+\frac{1}{p};-x^p\right).$$
\end{proof}

\begin{subsec}{\bf Proof of the Theorem \ref{thm1}.} \rm Let $0<x<y<1$, and $u=((x^t+y^t)/2)^{1/t}>x$. We denote ${\rm arcsin}(x),
{\rm artanh}(x),{\rm arctan}(x),{\rm arsinh}(x)$ by $g_i(x),\,i=1,2\ldots 4$ respectively, and define
$$g(x)=g_i(u)^t-\frac{g_i(x)^t+g_i(y)^t}{2}.$$
Differentiating with respect to $x$, we get $du/dx=(1/2)(x/u)^{t-1}$ and
\begin{eqnarray*}
g'(x)&=&\frac{1}{2}\,t\,g_i(x)^{t-1}\frac{d}{dx}(g_i(u))\left(\frac{x}{u}\right)^{t-1}
-\frac{1}{2}\,t\,g_i(x)^{t-1}\frac{d}{dx}(g_i(x))\\
&=&\frac{t}{2}x^{t-1}(f_i(u)-f_i(x)),
\end{eqnarray*}
where $$f_i(x)=\left(\frac{g_i(x)}{x}\right)^{t-1}\frac{d}{dx}(g_i(x)),\,i=1,2\ldots 4.$$
By Lemma \ref{lem1} $g'$ is positive and negative
for $f_{i=1,2}$ and $f_{i=3,4}$, respectively. This implies that
$$g(x)<(>)g(y)=0,$$
for $g_{i=1,2}$ and $g_{i=3,4}$, respectively. The case when $t=0$ follows from Lemma \ref{bvthm2}. This completes the proof.
\end{subsec}

\begin{lemma}\label{1a} For $p>1$ and $s\in(0,1)$, the function
$$\displaystyle f(p)=\left(\frac{\pi_p}{p}\right)^{-s}\frac{\left(p-\pi
\cot \left(\pi/p\right)\right) \csc(\pi/p)}{p^3}$$
is decreasing in $p\in(1,\infty)$,
\end{lemma}

\begin{proof} We have
$$f'(p)=\xi\left[2 p^2 (1-s)+\pi ^2 (1-s) \cot ^2\left(\frac{\pi }{p}\right)
-\pi  p (4-3s) \cot \left(\frac{\pi }{p}\right)+\pi
   ^2 \csc ^2\left(\frac{\pi }{p}\right)\right],$$
   which is negative, where $$\xi=-\frac{(2\pi)^{-s}}{p^3}\left(\frac{\csc(\pi/2)}{p^2}\right)^{1-s}.$$

\end{proof}

\begin{lemma}\label{lemku}\cite[Thm 2, p.151]{ku}
Let $J\subset\mathbb{R}$ be an open interval, and let $f:J\to \mathbb{R}$
be strictly monotonic function. Let $f^{-1}:f(J)\to J$ be the inverse to $f$ then
\begin{enumerate}
\item if $f$ is convex and increasing, then $f^{-1}$ is concave,
\item if $f$ is convex and decreasing, then $f^{-1}$ is convex,
\item if $f$ is concave and increasing, then $f^{-1}$ is convex,
\item if $f$ is concave and decreasing, then $f^{-1}$ is concave.
\end{enumerate}
\end{lemma}

\begin{lemma}\label{kuclem} For $m\geq 1,\,p>1$ and $x\in(0,1)$, the following functions
\begin{enumerate}
\item $h_1(x)=\displaystyle\left(\frac{{\rm sin}_p\,x}{x}\right)^{m-1}\frac{d}{dx}({\rm sin}_p\,x)$,\\
\item $h_2(x)=\displaystyle\left(\frac{{\rm tanh}_p\,x}{x}\right)^{m-1}\frac{d}{dx}({\rm tanh}_p\,x)$,\\
are decreasing in $x$, and
\item $h_3(x)=\displaystyle\left(\frac{{\rm cos}_p\,x}{x}\right)^{m-1}\frac{d}{dx}({\rm cos}_p\,x)$,\\
\item $h_4(x)=\displaystyle\left(\frac{{\rm tan}_p\,x}{x}\right)^{m-1}\frac{d}{dx}({\rm tan}_p\,x)$,\\
\item $h_5(x)=\displaystyle\left(\frac{{\rm sinh}_p\,x}{x}\right)^{m-1}\frac{d}{dx}({\rm sinh}_p\,x)$,\\
\end{enumerate}
are increasing in $x$.
\end{lemma}

\begin{proof} Let $f(x)={\rm arcsin}_p\,x,\,x\in(0,1)$. We get
$$f'(x)=\frac{1}{(1-x^p)^{1/p}}\,,$$
which is positive and increasing, hence $f$ is convex. Clearly
$\sin_p\,x$ is increasing, and by Lemma \ref{lemku} is concave, this
implies that $\frac{d}{dx}\sin_p\,x$ is decreasing, and $(\sin_p
x)/x$ is decreasing also by Lemma \ref{125}. Similarly we get that
$\frac{d}{dx}\tanh_p\,x$ is decreasing and $\frac{d}{dx}\cos_p\,x$,
$\frac{d}{dx}\tan_p\,x$, $\frac{d}{dx}\sinh_p\,x$ are increasing,
and the rest of proof follows from Lemma \ref{125}.
\end{proof}

\begin{subsec}{\bf Proof of the Theorem \ref{thm2}.} \rm The proof is similar to the proof of Theorem \ref{thm1}
and follows from Lemma \ref{kuclem}.
\end{subsec}

\begin{proposition}\label{prop} For $p,q>1$ and $t<1$, we have
$${\pi}_{M_t(p,q)}\leq M_t(\pi_p,{\pi}_q)\,.$$
\end{proposition}

\begin{proof} \rm Let $1<p<q<\infty$, and $w=((p^t+q^t)/2)^{1/t}>p$. We define
$$g(p)=({\pi}_p)^t-\frac{(\pi_p)^t+(\pi_q)^t}{2}.$$
Differentiating with respect to $p$, we get $dw/dp=(1/2)(p/w)^{t-1}$ and
\begin{eqnarray*}
g'(p)&=&\frac{1}{2}\,t\,({\pi}_p)^{t-1}\frac{d}{dx}({\pi}_w)\left(\frac{p}{w}\right)^{t-1}
-\frac{1}{2}\,t\,({\pi}_p)^{t-1}\frac{d}{dx}({\pi}_p)\\
&=&\frac{t}{2}p^{t-1}(f(w)-f(p)),
\end{eqnarray*}
where $$\displaystyle f(p)=\left(\frac{\pi_p}{p}\right)^{t-1}\frac{d}{dp}\pi_p\,.$$
Clearly $\pi_p$ is decreasing, hence $(\pi_p/p)^{t-1}$ is increasing for $t< 1$ and $d/dp(\pi_p)$ is increasing by the
proof of Lemma \cite[Lemma 3.6]{bv}. This implies that $f(p)$ is increasing, and it follows that $g$ is increasing.
Hence $g(p)<g(q)=0$. The case when $t=0$ follows from Lemma \ref{bvthm2}(5).
This completes the proof.
\end{proof}

The following lemma follows immediately from Lemma \ref{kuclem}.

\begin{lemma} For $p>1$ and $r,s\in(0,1)$ with $r\leq s$, we have
\begin{enumerate}
\item $\displaystyle\frac{\sin_p r}{r}\geq \frac{\sin_p s}{s},$\\
\item $\displaystyle\frac{\cos_p r}{r}\geq \frac{\cos_p s}{s},$\\
\item $\displaystyle\frac{\tan_p r}{r}\leq \frac{\tan_p s}{s},$\\
\item $\displaystyle\frac{\sinh_p r}{r}\leq \frac{\sinh_p s}{s},$\\
\item $\displaystyle\frac{\tanh_p r}{r}\geq \frac{\tanh_p s}{s}.$
\end{enumerate}
\end{lemma}


\end{document}